\documentclass{amsart}

\newtheorem{theorem}{Theorem}[section]
\newtheorem{lemma}[theorem]{Lemma}

\newtheorem{proposition}[theorem]{Proposition}

\theoremstyle{definition}
\newtheorem{definition}[theorem]{Definition} 

\theoremstyle{plain}

\newtheorem{question}[theorem]{Question}
\newtheorem*{claim}{Claim}

\usepackage[utf8]{inputenc}

\usepackage{amssymb}
\usepackage{mathrsfs}
\usepackage{bm}
\usepackage{cancel}
\usepackage{xcolor}
\usepackage{tikz}
\usetikzlibrary{positioning}
\usetikzlibrary{arrows}
\usepackage[margin=1.2in]{geometry}
\usepackage{pdflscape}
\usepackage{hyperref}
\hypersetup{
    colorlinks,
    citecolor=black,
    filecolor=black,
    linkcolor=black,
    urlcolor=black
}
\usepackage{pifont}

\usepackage{graphicx}
\graphicspath{ {./images/} }

\usepackage{wasysym}
\usepackage{mathdots}
\usepackage{array}

\usepackage{bbding}
\usepackage{dsfont}
\usepackage[backend=biber, style = alphabetic, maxbibnames = 10]{biblatex}
\usepackage{comment}

\usepackage{framed}
\usepackage{mdframed}

\title{Iterated jump noncomputability and compactness}
\author{Gavin Dooley}
\thanks{The author thanks his advisor, Peter Cholak, for the instruction, support, wisdom, and positive attitude that he supplies. He also thanks the Erwin Schrödinger International Institute for Mathematics and Physics at the University of Vienna for hosting the summer school and workshop titled ``Reverse Mathematics: New Paradigms,'' at which he learned many things that helped him improve many aspects of this paper.}
\address{Department of Mathematics, University of Notre Dame, USA}
\email{gdooley2@nd.edu}
\begin{document}

\maketitle$ $

\begin{abstract}
    We use reverse mathematics to analyze ``iterated jump'' versions of the following four principles: the atomic model theorem with subenumerable types ($\mathsf{AST}$), the diagonally noncomputable principle ($\mathsf{DNR}$), weak weak Kőnig's lemma ($\mathsf{WWKL}$), and weak Kőnig's lemma ($\mathsf{WKL}$). The logical relationships between these principles are summarized in Figure 1 and include, among other things, an infinite chain and an infinite antichain, the latter of which represents a strong form of non-linearity in terms of provability strength among ``natural'' combinatorial principles.
\end{abstract}

\section{Introduction}

Reverse mathematics and computable mathematics are two closely related ways of measuring the complexity of mathematical ideas. The former consists in formalizing theorems in the language of second-order arithmetic and comparing their logical strengths over a weak theory called $\mathsf{RCA}_0$. The latter identifies the extent to which mathematical constructions can be carried out algorithmically. (Specifically, it identifies which black boxes, or ``oracles,'' are necessary to carry them out.) Research has consistently uncovered connections between these two approaches, with many mathematical principles exhibiting equivalences with computability-theoretic properties. For a thorough treatment of contemporary reverse and computable mathematics, see \cite{DzhafarovMummert22}.

In this paper, we situate some well-known principles into a broader context by analyzing ``iterated jump'' versions of them. The logical relationships between these principles are summarized in Figure \ref{fig:iterjumps} and include, among other things, an infinite chain and an infinite antichain, the latter of which represents a strong form of non-linearity in terms of provability strength among ``natural'' combinatorial principles.

Section \ref{comptheochar} of this paper lists mathematical principles with known computability-theoretic characterizations, and Section \ref{logrel} discusses the logical relationships between these principles. Section \ref{iterjumpssection} introduces ``iterated jump'' versions of four of these principles and states the main results of this paper, which fully describe the logical relationships between them. Section \ref{proofssection} provides the proofs.

\begin{figure}
    \centering
    \begin{tikzpicture}
        \node[align = center] at (2, 8) (ACA) {$\mathsf{ACA}$};

        \node at (-0.5, 8.4) (dotsplaceholder0) {$ $};
        \node at (1.8, 7.2) (dotsplaceholder1) {$ $};

        \node[align = center] at (0.6, 7) (WKLplaceholder) {$ $};
        \node[align = center] at (-1.9, 6.2) (WWKLplaceholder) {$ $};
        \node[align = center] at (-4.4, 5.4) (RWWKLplaceholder) {$ $};
        \node[align = center] at (-6.9, 4.6) (ASTplaceholder) {$ $};

        \node[align = center] at (2.9, 5.8) (WKL2) {$\mathsf{WKL}^{(2)}$};
        \node[align = center] at (0.4, 5) (WWKL2) {$\mathsf{WWKL}^{(2)}$};
        \node[align = center] at (-2.1, 4.2) (RWWKL2) {$\mathsf{DNR}^{(2)}$};
        \node[align = center] at (-4.6, 3.4) (AST2) {$\mathsf{AST}^{(2)}$};
        
        \node[align = center] at (5.2, 4.6) (WKL1) {$\mathsf{WKL}^{(1)}$};
        \node[align = center] at (2.7, 3.8) (WWKL1) {$\mathsf{WWKL}^{(1)}$};
        \node[align = center] at (0.2, 3) (RWWKL1) {$\mathsf{DNR}^{(1)}$};
        \node[align = center] at (-2.3, 2.2) (AST1) {$\mathsf{AST}^{(1)}$};
       
        \node[align = center] at (7.5, 3.4) (WKL0) {$\mathsf{WKL}^{(0)}$};
        \node[align = center] at (5, 2.6) (WWKL0) {$\mathsf{WWKL}^{(0)}$};
        \node[align = center] at (2.5, 1.8) (RWWKL0) {$\mathsf{DNR}^{(0)}$};
        \node[align = center] at (0, 1) (AST0) {$\mathsf{AST}^{(0)}$};
        
        \node[align = center] at (0, -0.25) (RCA0) {$\mathsf{RCA}_0$};

        \draw[->, >=stealth, double, double distance = 1.5] (AST0) -- (RCA0);
        \draw[->, >=stealth, double, double distance = 1.5] (AST1) -- (AST0);
        \draw[->, >=stealth, double, double distance = 1.5] (AST2) -- (AST1);
        
        \draw[->, >=stealth, double, double distance = 1.5] (RWWKL0) -- (AST0);
        \draw[->, >=stealth, double, double distance = 1.5] (WWKL0) -- (RWWKL0);
        \draw[->, >=stealth, double, double distance = 1.5] (WKL0) -- (WWKL0);

        \draw[->, >=stealth, double, double distance = 1.5] (RWWKL1) -- (AST1);
        \draw[->, >=stealth, double, double distance = 1.5] (WWKL1) -- (RWWKL1);
        \draw[->, >=stealth, double, double distance = 1.5] (WKL1) -- (WWKL1);

        \draw[->, >=stealth, double, double distance = 1.5] (RWWKL2) -- (AST2);
        \draw[->, >=stealth, double, double distance = 1.5] (WWKL2) -- (RWWKL2);
        \draw[->, >=stealth, double, double distance = 1.5] (WKL2) -- (WWKL2);

        \draw[->, >=stealth, double, double distance = 1.5] (ACA) -- (WKL0);
        \draw[->, >=stealth, double, double distance = 1.5] (ACA) -- (WKL1);
        \draw[->, >=stealth, double, double distance = 1.5] (ACA) -- (WKL2);

        \draw[loosely dotted, very thick] (ASTplaceholder) -- (AST2);
        \draw[loosely dotted, very thick] (RWWKLplaceholder) -- (RWWKL2);
        \draw[loosely dotted, very thick] (WWKLplaceholder) -- (WWKL2);
        \draw[loosely dotted, very thick] (WKLplaceholder) -- (WKL2);
        
    \end{tikzpicture}
    \caption{Logical relationships between iterated jumps of principles. Double arrows indicate strict implications, and no additional arrows can be added. This diagram shows the logical relationships between these principles both over $\mathsf{RCA}_0$ and over $\omega$-models.}
    \label{fig:iterjumps}
\end{figure}

\section{Mathematical principles with computability-theoretic characterizations}
\label{comptheochar}

In this section, which may be superficially skimmed, we provide the statements of mathematical principles with known computability-theoretic equivalences and include references for proofs these equivalences.

When Friedman first introduced reverse mathematics, he immediately observed a correspondence between certain important mathematical principles and certain fundamental computability-theoretic notions. Define an \emph{$\omega$-model} to be a model of $\mathsf{RCA}_0$ whose first-order part is $\omega$ (i.e., whose first-order part is standard).

\begin{definition}$ $
    \begin{enumerate}
        \item \emph{Weak Kőnig's lemma} ($\mathsf{WKL}$) is the following statement: for every infinite tree $T \subseteq 2^{<\mathbb{N}}$, there exists a path $f \in 2^\mathbb{N}$ through $T$. Let $\mathsf{WKL}_0$ denote $\mathsf{RCA}_0 + \mathsf{WKL}$.
        \item The \emph{arithmetical comprehension axiom} ($\mathsf{ACA}$) is the following scheme: for every arithmetical formula $\varphi(n)$, there exists an $X$ such that $\forall n [n \in X \leftrightarrow \varphi(n)]$. Let $\mathsf{ACA}_0$ denote $\mathsf{RCA}_0 + \mathsf{ACA}$.
    \end{enumerate}
    
\end{definition}

\begin{theorem}[\cite{Friedman75} (pp. 237-238). See \cite{DzhafarovMummert22} (p. 96 and Thm. 5.5.3)]$ $
    \begin{enumerate}
        \item The second-order parts of $\omega$-models (of $\mathsf{RCA}_0$) are precisely the Turing ideals (i.e., the collections of sets downward closed under Turing reduction and closed under effective join).
        \item The second order parts of $\omega$-models of $\mathsf{WKL}_0$ are precisely the Scott ideals (i.e., the Turing ideals closed under the existence of PA degrees (i.e., the Turing ideals such that, for all $X$ in the ideal, there exists a $Y$ in the ideal that is of PA degree over $X$)).
        \item The second-order parts of $\omega$-models of $\mathsf{ACA}_0$ are precisely the jump ideals (i.e., the Turing ideals closed under the Turing jump).
    \end{enumerate}
\end{theorem}

Being of PA degree has many equivalent characterizations. For our purposes, it is convenient to take ``$Y$ is of PA degree over $X$'' (denoted ``$Y \gg X$'') to simply mean that $Y$ computes a path through any $X$-computable infinite binary tree, since this definition makes the equivalence of $\mathsf{WKL}$ with ``$\forall X \ \exists Y \ Y \gg X$'' most immediate.

Many other mathematical principles exhibit equivalences with computability-theoretic properties. Figure \ref{fig:compthychar} summarizes the relationships between these principles, which are subsequently detailed.

\begin{figure}
    \centering
    \begin{tikzpicture}
        \node[align = center] at (-2,11.5) (ACA) {$\mathsf{ACA}$ \\ $\forall X \exists Y \ Y \geq_\text{T} X'$};
        \node[align = center] at (-7.5,8.45) (WKL) {$\mathsf{WKL}$ \\ $\forall X \exists Y \ Y\gg X$};
        \node[align = center] at (-8.5, 6.5) (WWKL) {$\mathsf{WWKL}$ \\ $\forall X \exists Y \ Y\text{ $1$-random over } X$};
        \node[align = center] at (-7.5, 4.55) (RWWKL) {$\mathsf{RWWKL}$ \\ $\forall X \exists Y \ Y \text{ DNC over } X$};
        \node[align = center] at (-4, 9.3) (RRT22) {$\mathsf{RRT}^2_2$ \\ $\forall X \exists Y \ Y \text{ DNC over } X'$};
        \node[align = center] at (0, 9.3) (COH) {$\mathsf{COH}$ \\ $\forall X \exists Y \ Y' \gg X'$};
        \node[align = center] at (-2, 5.5) (Seq*TS1) {$\mathsf{Seq}^*(\mathsf{TS}^1)$ \\ $\forall X \exists Y \ Y' \text{ DNC over } X'$};
        \node[align = center] at (3, 8.2) (FIP) {$\mathsf{FIP}$ \\ $\forall X \exists Y \ Y \text{ $1$-generic over } X$};
        \node[align = center] at (-2, 3.1) (OPT) {$\mathsf{OPT}$ \\ $\forall X \exists Y \ Y \text{ hyperimmune over } X$};
        \node[align = center] at (-2, 1.1) (AST) {$\mathsf{AST}$ \\ $\forall X \exists Y \ Y \nleq_\text{T} X$};
        \node[align = center] at (-2, -0.35) (RCA0) {$\mathsf{RCA}_0$};

        \draw[->, >=stealth, double, double distance = 1.5] (AST) -- (RCA0);
        \draw[->, >=stealth, double, double distance = 1.5] (ACA) -- (WKL);
        \draw[->, >=stealth, double, double distance = 1.5] (ACA) -- (RRT22);
        \draw[->, >=stealth, double, double distance = 1.5] (ACA) -- (COH);
        \draw[->, >=stealth, double, double distance = 1.5] (ACA) -- (FIP);
        \draw[->, >=stealth, double, double distance = 1.5] (WKL) -- (WWKL);
        \draw[->, >=stealth, double, double distance = 1.5] (WWKL) -- (RWWKL);
        \draw[->, >=stealth, double, double distance = 1.5] (RWWKL) -- (AST);
        \draw[->, >=stealth, double, double distance = 1.5] (RRT22) -- (RWWKL);
        \draw[->, >=stealth, double, double distance = 1.5] (COH) -- (Seq*TS1);
        \draw[->, >=stealth, double, double distance = 1.5] (FIP) -- (OPT);
        \draw[->, >=stealth, double, double distance = 1.5] (OPT) -- (AST);
        \draw[->, >=stealth, double, double distance = 1.5] (Seq*TS1) -- (OPT);
        \draw[->, >=stealth, double, double distance = 1.5] (RRT22) -- (Seq*TS1);
        \draw[->] (WKL) -- (OPT) node[pos = 0.8, sloped] {$/$};
        \draw[->] (RRT22) -- (WWKL) node[pos = 0.4, sloped] {$/$};
        \draw[->] (COH) -- (RWWKL) node[pos = 0.5, sloped] {$/$};
        \draw[->] (FIP) -- (Seq*TS1) node[pos = 0.5, sloped] {$/$};
        \draw[->] (FIP) -- (RWWKL) node[pos = 0.65, sloped] {$/$};
        \draw[->] (RRT22) -- (COH) node[pos = 0.5, sloped] {$/$};
    \end{tikzpicture}
    \caption{Logical relationships between various principles with computability-theoretic characterizations. In the cases of $\mathsf{COH}$ and $\mathsf{\mathsf{Seq}^*(\mathsf{TS}^1)}$, these characterizations are only known to hold over $\omega$-models. In fact, the characterization for $\mathsf{COH}$ is specifically known not to hold over $\mathsf{RCA}_0$. (See \cite{Belanger2022}.) Double arrows indicate strict implications, and slashed arrows indicate nonimplications. This diagram shows both what is known over $\mathsf{RCA}_0$ and what is known over $\omega$-models.}
    \label{fig:compthychar}
\end{figure}

\begin{definition}
    \emph{Weak weak Kőnig's lemma} ($\mathsf{WWKL}$) is the following statement: for every tree $T \subseteq 2^{<\mathbb{N}}$ such that $\lim_n \frac{|\{ \sigma \in 2^n : \sigma \in T \}|}{2^n} >0$, there exists a path $f \in 2^\mathbb{N}$ through $T$.
\end{definition}

\begin{definition}[$1$-randomness]
    Given a finite set $U \in 2^{< \mathbb{N}}$, define $[[U]]$ to be the set of all $\mathcal{U} \in 2^{\mathbb{N}}$ such that $U \prec \mathcal{U}$. (I.e., define $[[U]]$ to be the basic open set generated by $U$ in Cantor space.) A sequence $\langle \mathcal{U}_n : n \in \omega \rangle$ of subsets of $2^\mathbb{N}$ is \emph{uniformly $\Sigma^0_1(D)$} if there is a uniformly $D$-c.e. sequence $\langle U_n : n \in \omega \rangle$ of subsets of $2^{< \mathbb{N}}$ such that $\mathcal{U}_n = [[U_n]]$ for all $n$. A \emph{Martin-Löf test over $D$} is a uniformly $\Sigma^0_1(D)$ sequence $\langle \mathcal{U}_n : n \in \omega \rangle$ of subsets of $2^\mathbb{N}$ such that $\mu(\mathcal{U}_n) \leq 2^{-n}$ for all $n$. A set $X$ \emph{passes} a Martin-Löf test $\langle \mathcal{U}_n : n \in \omega \rangle$ over $D$ if $X \notin \bigcap_n [[U_n]]$, where $\langle U_n : n \in \omega \rangle$ is the uniformly $D$-c.e. sequence generating $\langle \mathcal{U}_n : n \in \omega \rangle$. A set $X$ is \emph{$1$-random over $D$} if it passes every Martin-Löf test over $D$.
\end{definition}

\begin{theorem}[\cite{Kucera1985} (p. 248). See \cite{DzhafarovMummert22} (Cor. 9.11.3)]
    $\mathsf{RCA}_0 \vdash \mathsf{WWKL} \leftrightarrow \forall X \ \exists Y \ Y$ is $1$-random over $X$.
\end{theorem}

\begin{definition}
    Recall that $A \subseteq^*B$ means that $A \setminus B$ is finite. The \emph{cohesive principle} ($\mathsf{COH}$) is the following statement: for every sequence $\langle R_i : i \in \omega \rangle$ of subsets of $\mathbb{N}$, there exists an infinite set $S \subseteq \mathbb{N}$ such that, for each $i$, either $S \subseteq^* R_i$ or $S \subseteq^* \overline{R_i}$.
\end{definition}

\begin{theorem}[\cite{JockuschStephan93} (Thm. 2.1). See \cite{DzhafarovMummert22} (Thm. 8.4.13)] \label{JockuschStephen}
    The second order parts of $\omega$-models of $\mathsf{RCA}_0 + \mathsf{COH}$ are precisely those $\mathcal{I}$ such that, for each $X \in \mathcal{I}$, there exists a $Y \in \mathcal{I}$ such that $Y' \gg X'$.
\end{theorem}

\begin{definition}
    The \emph{omitting partial types principle} ($\mathsf{OPT}$) is the following statement: for every complete theory $T$ and every listing $\langle \Gamma_i : i \in \omega \rangle$ of partial types of $T$, there exists a model $\mathcal{B}$ of $T$ omitting all nonprincipal $\Gamma_i$.
\end{definition}

\begin{definition}[hyperimmunity]
    The principal function of an infinite set $X = \{ x_0 < x_1 < \dots \}$ is the function $p_X : \mathbb{N} \to X$ such that $p_X(n) = x_n$ for each $n$. Given $f, g : \mathbb{N} \to \mathbb{N}$, we say that $f$ \emph{dominates} $g$ if $f(n) \geq g(n)$ for all sufficiently large $n$. We say that an infinite set $X$ is \emph{hyperimmune over $D$} if no $D$-computable function dominates the principal function of $X$.
\end{definition}

\begin{theorem}[\cite{HirschfeldtShoreSlaman2009} (Thm. 5.7). See \cite{DzhafarovMummert22} (Thm. 9.10.12)]
   $\mathsf{RCA}_0 \vdash \mathsf{OPT} \leftrightarrow \forall X \ \exists Y \ Y$ is hyperimmune over $X$.
\end{theorem}

\begin{definition}
    For a theory $T$, a \emph{subenumeration} of the types of $T$ is a listing $\langle \Gamma_i : i \in \omega \rangle$ of partial types of $T$ such that, for each complete type $\Gamma$ of $T$, there is an $i$ such that $\Gamma$ and $\Gamma_i$ are equivalent. The \emph{atomic model theorem with subenumerable types} ($\mathsf{AST}$) is the following statement: for every complete theory $T$ and every subenumeration $\langle \Gamma_i : i \in \omega \rangle$ of the types of $T$, there exists an atomic model $\mathcal{B} \models T$.
\end{definition}

\begin{theorem}[\cite{HirschfeldtShoreSlaman2009} (Thm. 6.3). See \cite{DzhafarovMummert22} (Thm. 9.10.15)]
   $\mathsf{RCA}_0 \vdash \mathsf{AST} \leftrightarrow \forall X \ \exists Y \ Y \ngeq_\text{\emph{T}} X$.
\end{theorem}

\begin{definition}
    \emph{Two-bounded rainbow Ramsey's theorem for pairs} ($\mathsf{RRT}^2_2$) is the following statement: for every $c : [\mathbb{N}]^2 \to \mathbb{N}$ such that $|c^{-1}(z)| \leq 2$ for each $z \in \mathbb{N}$, there exists an infinite set $R$ such that $c$ is injective on $[R]^2$.
\end{definition}

\begin{definition}[diagonal noncomputability]
    A function $f : \mathbb{N} \to \mathbb{N}$ is \emph{DNC over $D$} if, for every $e \in \mathbb{N}$, if $\Phi^D_e(e) \downarrow$, then $f(e) \neq \Phi^D_e(e)$. A set $X$ is \emph{DNC over $D$} if $X$ computes a function that is DNC over $D$.
\end{definition}

\begin{theorem}[Joseph S. Miller, unpublished. See \cite{DzhafarovMummert22} (Thm. 9.4.13)]
    $\mathsf{RCA}_0 \vdash \mathsf{RRT}^2_2 \leftrightarrow \forall X \ \exists Y \ Y \text{ is DNC over } X'$.
\end{theorem}

\begin{definition}
    The \emph{sequential thin set theorem for singletons with finite errors} ($\mathsf{Seq}^*(\mathsf{TS}^1)$) is the following statement: for every sequence $\langle f_i : i \in \omega \rangle$ of functions $\mathbb{N} \to \mathbb{N}$, there exists an infinite set $H \subseteq \mathbb{N}$ such that, for each $i$, there is some $H_i =^* H$ with $f(H_i) \neq \mathbb{N}$.
\end{definition}

\begin{theorem}[\cite{Patey2016} (Lems. 6.1 and 6.2). See \cite{patey_lowness_avoidance} (Ex. 10.1.2)] \label{Patey16}
    The second order parts of $\omega$-models of $\mathsf{RCA}_0 + \mathsf{Seq}^*(\mathsf{TS}^1)$ are precisely those $\mathcal{I}$ such that, for each $X \in \mathcal{I}$, there exists a $Y \in \mathcal{I}$ such that $Y' \text{ is DNC over } X'$.
\end{theorem}

\begin{definition}
    Say that a set $H \subseteq \mathbb{N}$ is \emph{homogeneous} for a string $\sigma \in 2^{< \mathbb{N}}$ if there exists some $c < 2$ such that, for all $i \in H$ with $i < |\sigma|$, we have $\sigma(i) = c$. We say that $H$ is \emph{homogeneous} for a tree $T \subseteq 2^{< \mathbb{N}}$ if the subtree $\{ \sigma \in T : H \text{ is homogeneous for } \sigma \}$ is infinite. \emph{Ramsey-type weak weak Kőnig's lemma} ($\mathsf{RWWKL}$) is the following statement: for every tree $T \subseteq 2^{< \mathbb{N}}$ such that $\lim_n \frac{|\{ \sigma \in 2^n : \sigma \in T \}|}{2^n} > 0$, there exists a set $H$ that is homogeneous for $T$.
\end{definition}

\begin{theorem}[\cite{BienvenuPateyShafer2017} (Thm. 3.4)]
    $\mathsf{RCA}_0 \vdash \mathsf{RWWKL} \leftrightarrow \forall X \ \exists Y \ Y \text{ is DNC over } X$.
\end{theorem}

\begin{definition}
    Say that a family of sets  $\vec{X} = \langle X_i : i \in \omega \rangle$ is \emph{non-trivial} if $X_i \neq \emptyset$ for some $i$. Say that $\vec{X}$ has the \emph{finite intersection property} if, for every nonempty finite set $F \subseteq \omega$, $\bigcap_{i \in F} X_i \neq \emptyset$. The \emph{finite intersection principle} ($\mathsf{FIP}$) is the following statement: for every nontrivial family of sets $\vec{X}$, there exists a subfamily $\vec{Y} \subseteq \vec{X}$ satisfying the finite intersection property that is maximal with respect to having this property.
\end{definition}

\begin{definition}[$1$-genericity]
    Given an infinite set $X \subseteq \mathbb{N}$ and a set $V \subseteq 2^{< \mathbb{N}}$, we say that $X$ \emph{forces} $V$ if there exists some $\sigma \prec X$ such that either $\sigma \in V$ or, for all $\rho \succ \sigma$, $\rho \notin V$. We say that $X$ is \emph{$1$-generic over $D$} if $X$ forces $W_e^D$ for each $e$.
\end{definition}

\begin{theorem}[\cite{CholakDowneyIgusa2017} (Thm. 1.6). See \cite{DzhafarovMummert22} (Thm. 9.10.23)]
    $\mathsf{RCA}_0 \vdash \mathsf{FIP} \leftrightarrow \forall X \ \exists Y \ Y$ is $1$-generic over $X$.
\end{theorem}

\section{Logical relationships between these principles}
\label{logrel}

In this section, we list all known implications and nonimplications between the principles listed above and provide a proof for one new implication. We also list the three remaining open questions regarding the logical relationships between these principles.

We write $\mathsf{Q} \nleqslant_\omega \mathsf{P}$ if every $\omega$-model satisfying $\mathsf{P}$ also satisfies $\mathsf{Q}$. The following implications and nonimplications appear in the literature.
\begin{theorem}$ $
    \begin{enumerate}
        \item $\mathsf{RCA}_0 \vdash \mathsf{ACA} \to \mathsf{WKL}$ \emph{\cite{Friedman75}};
        \item $\mathsf{RCA}_0 \vdash \mathsf{ACA} \to \mathsf{RRT}^2_2$ \emph{\cite{Specker1971}, \cite{CsimaMileti2009}};
        \item $\mathsf{RCA}_0 \vdash \mathsf{ACA} \to \mathsf{COH}$ \emph{\cite{Specker1971}, [Jockusch and Lempp, unpublished], \cite{Mileti2004}};
        \item $\mathsf{RCA}_0 \vdash \mathsf{ACA} \to \mathsf{FIP}$ \emph{\cite{DzhafarovMummert2013}};
        \item $\mathsf{RCA}_0 \vdash \mathsf{WKL} \to \mathsf{WWKL}$ and $\mathsf{WKL} \nleqslant_\omega \mathsf{WWKL} $ \emph{\cite{YuSimpson1990}};
        \item $\mathsf{RCA}_0 \vdash \mathsf{WWKL} \to  \mathsf{RWWKL}$ and $\mathsf{WWKL} \nleqslant_\omega \mathsf{RWWKL} $ \emph{\cite{GiustoSimpson2000}, \cite{ASKHLS}, \cite{FloodTowsner2016}};
        \item $\mathsf{RCA}_0 \vdash \mathsf{RWWKL} \to \mathsf{AST}$ \emph{\cite{HirschfeldtShoreSlaman2009}, \cite{BienvenuPateyShafer2017}};
        \item $\mathsf{RCA}_0 \vdash \mathsf{RRT^2_2} \to \mathsf{RWWKL}$ and $\mathsf{RRT}^2_2 \nleqslant_\omega \mathsf{RWWKL}$ \emph{[Joseph S. Miller, unpublished], \cite{FloodTowsner2016}};
        \item $\mathsf{RCA}_0 \vdash \mathsf{COH} \to \mathsf{Seq}^*(\mathsf{TS}^1)$ and $\mathsf{COH} \nleqslant_\omega \mathsf{Seq}^*(\mathsf{TS}^1)$ \emph{\cite{Patey2016}};
        \item $\mathsf{RCA}_0 \vdash \mathsf{COH} \to \mathsf{OPT}$ \emph{\cite{HirschfeldtShoreSlaman2009}};
        \item $\mathsf{RCA}_0 \vdash \mathsf{RRT}^2_2 \to \mathsf{OPT}$ \emph{\cite{CsimaMileti2009}, \cite{HirschfeldtShoreSlaman2009}};
        \item $\mathsf{RCA}_0 \vdash \mathsf{RRT^2_2} \to \mathsf{Seq}^*(\mathsf{TS}^1)$ \emph{\cite{Patey2016}};
        \item $\mathsf{RCA}_0 \vdash \mathsf{Seq}^*(\mathsf{TS}^1) \to \mathsf{AST}$ \emph{\cite{Patey2016}};
        \item $\mathsf{RCA}_0 \vdash \mathsf{FIP} \to \mathsf{OPT}$ and $\mathsf{FIP} \nleqslant_\omega \mathsf{OPT}$ \emph{\cite{Sacks1963}, \cite{MillerMartin1968}, \cite{HirschfeldtShoreSlaman2009}, \cite{DzhafarovMummert2013}, \cite{CholakDowneyIgusa2017}};
        \item $\mathsf{RCA}_0 \vdash \mathsf{OPT} \to \mathsf{AST}$ \emph{\cite{HirschfeldtShoreSlaman2009}};
        \item $\mathsf{Seq}^*(\mathsf{TS}^1) \nleqslant_\omega \mathsf{WKL}$ \emph{\cite{JockuschSoare72}};
        \item $\mathsf{OPT} \nleqslant_\omega \mathsf{WKL}$ \emph{\cite{HirschfeldtShoreSlaman2009}};
        \item $\mathsf{WWKL} \nleqslant_\omega \mathsf{RRT}^2_2$ \emph{\cite{CsimaMileti2009}, \cite{Liu2015}};
        \item $\mathsf{COH} \nleqslant_\omega \mathsf{RRT}^2_2$ \emph{\cite{vanlambalgen1987}, \cite{CsimaMileti2009}, \cite{patey:tel-01888675}};
        \item $\mathsf{RWWKL} \nleqslant_\omega \mathsf{COH}$ \emph{\cite{HJKLS}, \cite{BienvenuPateyShafer2017}};
        \item $\mathsf{Seq^*(TS^1)} \nleqslant_\omega \mathsf{FIP}$ \emph{\cite{Patey2016}, \cite{CholakDowneyIgusa2017}, \cite{patey_lowness_avoidance}};
        \item $\mathsf{RWWKL} \nleqslant_\omega \mathsf{FIP}$ \emph{\cite{HirschfeldtShore2007}, \cite{Conidis2008}, \cite{HirschfeldtShoreSlaman2009}, \cite{DzhafarovMummert2013}, \cite{BienvenuPateyShafer2017}};
        \item $\mathsf{Seq}^*(\mathsf{TS}^1) \nleqslant_\omega \mathsf{OPT}$  \emph{\cite{HirschfeldtShoreSlaman2009}, \cite{Patey2016}};
        \item $\mathsf{AST} \nleqslant_\omega \mathsf{RCA}_0$ \emph{\cite{HirschfeldtShoreSlaman2009}}.
    \end{enumerate}
\end{theorem}

We simultaneously strengthen Parts (10) and (11) above as follows.

\begin{proposition}
    $\mathsf{RCA}_0 \vdash \mathsf{Seq}^*(\mathsf{TS}^1) \to \mathsf{OPT}$.
\end{proposition}

\begin{proof}
    We gently modify the proof of Corollary 5.9 given in \cite{HirschfeldtShoreSlaman2009}. Let $s^e_n$ be the least $s$ such that $\Phi_e \upharpoonright (n +1)[s] \downarrow$, if one exists. For each $e$, we have $s_n^e \leq s_{n+1}^e$. Assume without loss of generality that $\Phi_e(n) \leq s_n^e$. Inductively define $r_0^e := s_0^e$ and $r_{n+1}^e := s^e_{\sum_{i=0}^n r_i^e}$. Define $R_n^e := \sum_{i = 0}^n r_i^e$.
    
    For each $e$, define $f_e : \mathbb{N} \to \mathbb{N}$ as 
    $$f_e(x):=
    \begin{cases}
        0 & 0 \leq x < R_0^e \\
        0 & R_0^e \leq x < R_1^e \\
        1 & R_1^e \leq x < R_2^e \\
        0 & R_2^e \leq x < R_3^e \\
        1 & R_3^e \leq x < R_4^e \\
        2 & R_4^e \leq x < R_5^e \\
        0 & R_5^e \leq x < R_6^e \\
        1 & R_6^e \leq x < R_7^e \\
        2 & R_7^e \leq x < R_8^e \\
        3 & R_8^e \leq x < R_9^e \\
        \vdots & \quad \quad \ \ \ \vdots
    \end{cases}$$
    if $\Phi_e$ is total and as
    $$f_e(x) :=
    \begin{cases}
        0 & 0 \leq x < R_0^e \\
        0 & R_0^e \leq x < R_1^e \\
        1 & R_1^e \leq x < R_2^e \\
        0 & R_2^e \leq x < R_3^e \\
        1 & R_3^e \leq x < R_4^e \\
        2 & R_4^e \leq x < R_5^e \\
        \vdots & \quad \quad \ \ \ \vdots \\
        0 & R_{n-1}^e \leq x
    \end{cases}
    $$
    if $n$ is the least number such that $r^e_n$ is undefined. The sequence $\langle f_e : e \in \omega \rangle$ is computable. Let $H$ be a $\mathsf{Seq}^*(\mathsf{TS}^1)$-solution to $\langle f_e : e \in \omega \rangle$. Let $g$ be the $H$-computable function defined by letting $g(n)$ be the least $m$ such that the sequence giving the characteristic function of $H$ has at least $n$ zeroes and at least $n$ ones among its first $m$ bits. We claim that $g$ is not dominated by any total computable function. Let $\Phi_e$ be total. Since $H$ is a $\mathsf{Seq}^*(\mathsf{TS}^1)$-solution to $\langle f_e : e \in \omega \rangle$, there exists some $k \in \mathbb{N}$ and some $H_0 =^* H$ such that $k \notin f_e(H_0)$. By the definition of $f_e$, for any $m$, we can find $y > x > m$ such that $x, y \in H$, there is no element of $H$ strictly between $x$ and $y$, and for some $n$ we have $x < \sum_{i=0}^n r_i^e$ and $y \geq \sum_{i=0}^{n+1} r_i^e$. Then $g(x) > y > r_{n+1}^e = s^e_{\sum_{i=0}^n r_i^e} \geq s_x^e \geq \Phi_e(x)$. Taking $m$ large enough, we see that $\Phi_e$ does not dominate $g$, which means that $H$ is hyperimmune.
\end{proof}

These implications and nonimplications leave open only the following questions about these principles.

\begin{question}$ $
    \begin{enumerate}
        \item Does $\mathsf{RCA}_0 \vdash \mathsf{RRT}^2_2 \to \mathsf{FIP}$?
        \item Does $\mathsf{RCA}_0 \vdash \mathsf{COH} \to \mathsf{FIP}$?
        \item Does $\mathsf{RCA}_0 \vdash \mathsf{Seq}^*(\mathsf{TS}^1) \to \mathsf{FIP}$?
        \item Do any of the above possible implications hold over $\omega$-models?
    \end{enumerate}
\end{question}

One partial step towards answering the second of these has been taken—in Corollary 12.8 of \cite{BrattkaHendtlassKreuzer2017}, Brattka, Hendtlass and Kreuzer prove that the principle $1\text{-}\mathsf{GEN}$ (which is equivalent to $\mathsf{FIP}$ over $\mathsf{RCA}_0$) cannot be reduced to $\mathsf{COH}$ via a uniform (i.e., Weihrauch) reduction. For a further analysis of computability-theoretic principles from the standpoint of Weihrauch reducibility, see \cite{BrattkaHendtlassKreuzer2017}.

\section{Iterated jumps of \texorpdfstring{$\mathsf{AST}$}{AST}, \texorpdfstring{$\mathsf{DNR}$}{DNR}, 
\texorpdfstring{$\mathsf{WWKL}$}{WWKL}, and \texorpdfstring{$\mathsf{WKL}$}{WKL}} 
\label{iterjumpssection}

\begin{definition} Fix $m \in \omega$.
    \begin{enumerate}
        \item Let $\mathsf{AST}^{(m)}$ denote the following statement: $\forall X \ \exists Y \ Y^{(m)} \nleq_\text{T} X^{(m)}$.
        \item Let $\mathsf{DNR}^{(m)}$ denote the following statement: $\forall X \ \exists Y \ Y^{(m)} \ \text{DNC over } X^{(m)}$.
        \item Let $\mathsf{WWKL}^{(m)}$ denote the following statement: $\forall X \ \exists Y \ Y^{(m)} \text{ computes a $1$-random over } X^{(m)}$.
        \item Let $\mathsf{WKL}^{(m)}$ denote the following statement: $\forall X \ \exists Y \ Y^{(m)} \gg X^{(m)}$.
    \end{enumerate}    
\end{definition}
We include the word ``computes'' in the definition of $\mathsf{WWKL}^{(m)}$ since \emph{computing} a $1$-random is a property of degrees (whereas \emph{being} a $1$-random is a property of sets). We are already aquatinted with these principles in the case where $m = 0$:
\begin{itemize}
    \item $\mathsf{AST}^{(0)} \equiv_{\mathsf{RCA}_0} \mathsf{AST}$,
    \item $\mathsf{DNR}^{(0)} \equiv_{\mathsf{RCA}_0} \mathsf{RWWKL}$,
    \item $\mathsf{WWKL}^{(0)} \equiv_{\mathsf{RCA}_0} \mathsf{WWKL}$,
    \item $\mathsf{WKL}^{(0)} \equiv_{\mathsf{RCA}_0} \mathsf{WKL}$.
\end{itemize}
Additionally, we have seen two of these principles in the case where $m = 1$:
\begin{itemize}
    \item $\mathsf{RCA}_0 \vdash \mathsf{DNR}^{(1)} \leftrightarrow \mathsf{Seq}^*(\mathsf{TS})$,
    \item $\mathsf{RCA}_0 \vdash \mathsf{WKL}^{(1)} \leftrightarrow \mathsf{COH}$.
\end{itemize}

Just as $\mathsf{AST}$ is the weakest natural (i.e., relativizing) principle not true in the $\omega$-model consisting entirely of computable sets; $\mathsf{AST}^{(m)}$ is the weakest natural principle not true in any $\omega$-model consisting entirely of low$_m$ sets. The principle $\mathsf{WKL}^{(m)}$ appears in \cite{FKWY2024} (as $\Delta^0_{m+1}\text{-}\mathsf{WKL}$) and in \cite{Belanger2022} (as $(m+1)\text{-}\mathsf{WKL}$).

Let us begin by observing the immediate logical relationships between these principles. The halting set is of PA degree. By \cite{MartinLof66}, every set that is of PA degree computes a $1$-random. By \cite{GiustoSimpson2000}, every $1$-random computes a DNC function. Every DNC set is in particular noncomputable. Since all of these results relativize, the following implications hold.

\begin{proposition}
    Let $m \in \mathbb{N}$. We have
    \begin{itemize}
        \item $\mathsf{RCA}_0 \vdash \mathsf{ACA} \to \mathsf{WKL}^{(m)}$,
        \item $\mathsf{RCA}_0 \vdash \mathsf{WKL}^{(m)} \to \mathsf{WWKL}^{(m)}$,
        \item $\mathsf{RCA}_0 \vdash \mathsf{WWKL}^{(m)} \to \mathsf{DNR}^{(m)}$, and
        \item $\mathsf{RCA}_0 \vdash \mathsf{DNR}^{(m)} \to \mathsf{AST}^{(m)}$.
    \end{itemize}
\end{proposition}

Since $Y \leq_\text{T} X$ implies $Y' \leq_\text{T} X'$, the following implication also holds.

\begin{proposition} Let $m < n$. We have $\mathsf{RCA}_0 \vdash \mathsf{AST}^{(n)} \to \mathsf{AST}^{(m)}$.
\end{proposition}

In \cite{FKWY2024}, it is proven that $\mathsf{RCA}_0^* \vdash \mathsf{WKL}^{(m)} \to \mathsf{B}\Sigma^0_{m+1}$. In personal communcation, Kołodziejczyk pointed out that, since the tree constructed in that proof has positive measure, we actually have that $\mathsf{RCA}_0^* \vdash \mathsf{WWKL}^{(m)} \to \mathsf{B}\Sigma^0_{m+1}$. In \cite{FKWY2024}, it is also proven that $\mathsf{WKL}^{(m)}$ is $\Pi^1_1$-conservative over $\mathsf{RCA}_0^* + \mathsf{B}\Sigma^0_{m+1}$, which implies that the implication stated above is as strong as possible. On the other hand, since any sufficiently Cohen generic is not low$_m$, one can prove that $\mathsf{AST}^{(m)}$ is $\Pi^1_1$-conservative over $\mathsf{RCA}_0$. These results together give us a thorough understanding of the first-order parts of $\mathsf{WKL}^{(m)}$, $\mathsf{WWKL}^{(m)}$, and $\mathsf{AST}^{(m)}$, but it would be nice to know more about the first-order part of $\mathsf{DNR}^{(m)}$.

Additionally, Kołodziejczyk has made the interesting observation that Theorem 3.9 from \cite{FKWY2024} implies that, for any $m$ and any $n$, we have $\mathsf{RCA}_0 \vdash \mathsf{WKL}^{(m)} + \neg \mathsf{I}\Sigma^0_{m+1} \to \mathsf{AST}^{(n)}$, an implication that contrasts with Theorem \ref{ASTnvsWKLm} below.

In \cite{Belanger2022}, Belanger points out that the principles $\mathsf{WKL}^{(m)}$ are logically independent from each other over $\omega$-models of $\mathsf{RCA}_0$, an assertion that we prove below and expand on via the following four separations, which give us the complete understanding summarized above in Figure \ref{fig:iterjumps}.

\begin{theorem} \label{WKLmvsWWKLm}
    Let $m \in \mathbb{N}$. We have $\mathsf{WKL}^{(m)} \nleqslant_\omega \mathsf{WWKL}^{(m)}$.
\end{theorem}

\begin{theorem} \label{WWKLmvsDNRm}
    Let $m \in \mathbb{N}$. We have $\mathsf{WWKL}^{(m)} \nleqslant_\omega \mathsf{DNR}^{(m)} $.
\end{theorem}

\begin{theorem}
\label{ASTnvsWKLm}
    Let $m < n$. We have $\mathsf{AST}^{(n)} \nleqslant_\omega \mathsf{WKL}^{(m)}$.
\end{theorem}

\begin{theorem}
\label{DNRmvsWKLn}
    Let $m < n$. We have $\mathsf{DNR}^{(m)} \nleqslant_\omega \mathsf{WKL}^{(n)} $.
\end{theorem}

\section{Proofs of separations}
\label{proofssection}

\subsection{Base cases: \texorpdfstring{$m = 0$ and $n = 1$}{m=0 and n=1}}

All four of these separations are already known in the case where $m = 0$ and $n = 1$. However, their earliest and best-known proofs do not necessarily lend themselves to the kind of generalization that we need. To obtain our generalizations, we expand on proofs of these separations in the following abstract framework, which was first isolated in \cite{patey:tel-01888675}, but the idea underlying which had previously appeared implicitly in various places, such as \cite{lst}.

An \emph{instance-solution problem} is a function $\mathsf{P} : \mathcal{P}(\omega) \to \mathcal{P}(\mathcal{P}(\omega))$. We refer to sets $X \in \mathcal{P}(\omega)$ as \emph{instances of $\mathsf{P}$} and sets $Y \in \mathsf{P}(X)$ as \emph{solutions to $X$}. Any mathematical principle of the form ``for all...there exists...'' can be regarded as an instance-solution problem in a natural way.

\begin{definition}
    A class $\mathcal{W} \subseteq 2^\omega$ is called a \emph{weakness property} if it is downward closed under Turing reduction. An instance-solution problem $\mathsf{P}$ is said to \emph{preserve} a weakness property $\mathcal{W}$ if, for every $Z \in \mathcal{W}$ and every $Z$-computable instance $X \in \mathcal{W}$, there is a solution $Y$ to $X$ such that $Z \oplus Y \in \mathcal{W}$.
\end{definition}

\begin{lemma}[\cite{patey:tel-01888675} (Lem. 3.4.2). See \cite{DzhafarovMummert22} (Thm. 4.6.13)] \label{separation-thm}
    Let $\mathsf{P}$ and $\mathsf{Q}$ be instance-solution problems and $\mathcal{W}$ be a weakness property. If $\mathsf{P}$ preserves $\mathcal{W}$ and $\mathsf{Q}$ does not, then $\mathsf{Q} \nleqslant_\omega \mathsf{P}$.
\end{lemma}

\begin{proposition}
    \label{WKL0vsWWKL0}
    $\mathsf{WKL}^{(0)} \nleqslant_\omega \mathsf{WWKL}^{(0)}$.
\end{proposition}

\begin{proof}
    To prove this separation, one can use the obvious weakness property; take $\mathcal{W} := \{ S : S \not\gg \emptyset \}$. By definition, $\mathsf{WKL}^{(0)}$ does not preserve $\mathcal{W}$. In slides 79-82 of \cite{Patey2025}, Patey proves that any set sufficiently generic for Solovay forcing is not of PA degree (even when combined with a a fixed set $Z$ not of PA degree). In Theorem II.3.5 of \cite{Kautz1991}, Kautz proves that any set sufficiently generic for Solovay forcing is $1$-random. Combining these results proves that $\mathsf{WWKL}^{(0)}$ does not preserve $\mathcal{W}$, as needed.
\end{proof}

\begin{proposition}
    \label{WWKL0vsDNR0}
    $\mathsf{WWKL}^{(0)} \nleqslant_\omega \mathsf{DNR}^{(0)}$.
\end{proposition}

To prove this separation, it is actually convenient to use a weakness property other than the obvious one. For any function $h : \omega \to \omega$, say that a function $g$ is \emph{$h$-DNC over $D$} if $g$ is DNC over $D$ and $g(x) < h(x)$ for all $x$. Say that a set $X$ is \emph{$h$-DNC over $D$} if $X$ computes a function that is $h$-DNC over $D$.

\begin{proof}
    Take $\mathcal{W} := \{ S : S \text{ is not }(x \mapsto 2^x)\text{-DNC over } \emptyset \}$. By Theorem 1.4 from \cite{ASKHLS}, $\mathsf{WWKL}^{(0)}$ does not preserve $\mathcal{W}$. By Theorem 3.7 from \cite{KhanMiller17}, $\mathsf{DNR}^{(0)}$ preserves $\mathcal{W}$, as needed.
\end{proof}

\begin{proposition}
    \label{AST1vsWKL0}
    $\mathsf{AST}^{(1)} \nleqslant_\omega \mathsf{WKL}^{(0)}$.
\end{proposition}

\begin{proof}
    This separation can be proven with the obvious weakness property; take $\mathcal{W} := \{ S : S' \leq_\text{T} \emptyset' \}$. By definition, $\mathsf{AST}^{(1)}$ does not preserve $\mathcal{W}$. By an appropriate relativization of \cite{JockuschSoare72}'s low basis theorem, $\mathsf{WKL}^{(0)}$ preserves $\mathcal{W}$, as needed.
\end{proof}

\begin{proposition}
    \label{DNR0vsWKL1}
    $\mathsf{DNR}^{(0)} \nleqslant_\omega \mathsf{WKL}^{(1)}$.
\end{proposition}

\begin{proof}
    This separation can also be proved with the obvious weakness property; take $\mathcal{W} := \{ S : S \text{ is not DNC over } \emptyset \}$. By definition, $\mathsf{DNR}^{(0)}$ does not preserve $\mathcal{W}$. By Exercise 5.8.6 from \cite{patey_lowness_avoidance}, $\mathsf{WKL}^{(1)}$ does not preserve $\mathcal{W}$, as needed.
\end{proof}

In each of the above proofs, the preservation of a weakness property corresponds to an abstract degree-theoretic fact. In all four cases, these abstract degree-theoretic facts remain true when relativized to an arbitrary degree $D$, and we will take advantage of this when proving these separations in their general cases. These four theorems are stated in their relativized forms below.

\begin{theorem}[\cite{Patey2025} (slides 79-82)]
    \label{thmforWKL0vsWWWKL0}
    For any $D$ and any $Z$ such that $Z \not\gg D$, there exists a $G$ such that $G$ is $1$-random over $D$ and $(Z \oplus G) \not\gg D$.
\end{theorem}

\begin{theorem}[\cite{KhanMiller17} (Thm. 3.7)]
    \label{thmforWWKL0vsDNR0}
    For any $D$ and any $Z$ such that $Z$ is not $(x \mapsto 2^x)$-DNC over $D$, there exists a $G$ such that $G$ is DNC over $D$ and $(Z \oplus G)$ is not $(x \mapsto 2^x)$-DNC over $D$.
\end{theorem}

\begin{theorem}[Low basis theorem; \cite{JockuschSoare72} (Thm. 2.1). See \cite{DzhafarovMummert22} (Thm. 2.8.18)]
\label{lowbasistheorem}
    For any $D$ and any $Z$ such that $Z' \leq_\text{T} D'$, there exists a $G$ such that $G \gg D$ and $(Z \oplus G)' \leq_\text{T} D'$.
\end{theorem}

\begin{theorem}[\cite{patey_lowness_avoidance} (Ex. 5.8.6)]
    \label{thmforDNR0vsWKL1}
    For any $D$, and any $Z$ such that $Z$ is not DNC over $D$, and any $A$, there exists a $G$ such that $G' \geq_\text{T} A$ and $(Z \oplus G)$ is not DNC over $D$.
\end{theorem}

Notice that this Theorem \ref{thmforDNR0vsWKL1} is stronger than what we needed to prove Proposition \ref{AST1vsWKL0}, since this theorem can be used to produce a $G$ whose jump computes any arbitrary fixed set $A$ (which in particular could be chosen to be of PA degree over $D'$).

Our proofs of Theorems \ref{WKLmvsWWKLm}, \ref{WWKLmvsDNRm}, and \ref{DNRmvsWKLn} (i.e., those in which a principle is separated from a one whose ``iterated jump exponent'' is no larger than that of the first principle) are identical in form to one another. The proof of Proposition \ref{ASTnvsWKLm}, in which a principle is separated from one with strictly larger ``iterated jump exponent,'' is shorter than the other three and uses Friedberg's jump inversion theorem. Let us begin with this one.

\subsection{Proof that \texorpdfstring{$\mathsf{AST}^{(n)} \nleqslant_\omega \mathsf{WKL}^{(m)} $}{AST(n) ⩽̸ω WKL(m)} for \texorpdfstring{$m < n$}{m<n}}

To prove this separation, we combine the low basis theorem (which proves the separation in the case where $m = 0$ and $n = 1$) with Friedberg's jump inversion theorem.

\begin{theorem}[Jump inversion theorem; \cite{Friedberg57} (p. 159). See \cite{Soare16} (Thm. 6.4.4)]
    For any $D$ and any $A \geq_\text{\emph{T}} D'$, there exists a $G \geq_\text{\emph{T}} D$ such that $G' \equiv_\text{\emph{T}} A$.
\end{theorem}

\begin{proof}[Proof of Theorem \ref{AST1vsWKL0}]
    Let $\mathcal{W} = \{ S : S^{(n)} \leq_\text{T} \emptyset^{(n)} \}$. By definition, $\mathsf{AST}^{(n)}$ does not preserve $\mathcal{W}$. We must prove that $\mathsf{WKL}^{(m)}$ preserves $\mathcal{W}$. Let $Z \in \mathcal{W}$ and $X \leq_\text{T} Z$. We must find a $\mathsf{WKL}^{(m)}$-solution $Y$ to $X$ such that $Z \oplus Y \in \mathcal{W}$. By the \hyperref[lowbasistheorem]{low basis theorem}, there exists some $G \gg Z^{(m)}$ such that $G' \equiv_\text{T} Z^{(m+1)}$. Using Friedberg's jump inversion theorem $m$ times (relativized to each of $Z^{(m-1)}, Z^{(m-2)}, \dots Z$ in that order), we can find a $Y \geq_\text{T} Z$ such that $Y^{(m)} \equiv_\text{T} G$. Since $G \gg Z^{(m)}$ and $Z^{(m)} \geq_\text{T} X^{(m)}$, we have $Y^{(m)} \gg X^{(m)}$, as needed. Additionally, since $Y \geq_\text{T} Z$ and $Y^{(m+1)} \equiv_\text{T} Z^{(m+1)}$, we have $(Z \oplus Y)^{(n)} \equiv_\text{T} Y^{(n)} \equiv_\text{T} Z^{(n)}$. Since $Z^{(n)} \leq_\text{T} \emptyset^{(n)}$, it follows that $Y^{(n)} \leq_\text{T} \emptyset^{(n)}$, as needed.
\end{proof}

\subsection{Proofs that \texorpdfstring{$\mathsf{WKL}^{(m)} \nleqslant_\omega \mathsf{WWKL}^{(m)}$}{WKL(m) ⩽̸ω WWKL(m)}, \texorpdfstring{$\mathsf{WWKL}^{(m)} \nleqslant_\omega \mathsf{DNR}^{(m)}$}{WWKL(m) ⩽̸ω DNR(m)}, and \texorpdfstring{$\mathsf{DNR}^{(m)} \nleqslant_\omega \mathsf{WKL}^{(n)}$}{DNR(m) ⩽̸ω WKL(n)} for \texorpdfstring{$m<n$}{m<n}}

To prove these separations, it will suffice to prove the following ``iterated jump'' versions of the relevant degree-theoretic facts. (In the case of the second separation, proving this degree-theoretic fact suffices because \cite{ASKHLS}'s proof that $\mathsf{WWKL}^{(0)}$ does not preserve $\mathcal{W} = \{ S : S \text{ is not } (x \mapsto 2^x)\text{-DNC over } \emptyset \}$ immediately generalizes to an analogous statement regarding $\mathsf{WWKL}^{(m)}$.)

\begin{theorem}
    \label{thmforWKLmvsWWKLm}
    For any $D$ and any $Z$ such that $Z^{(m)} \not\gg D^{(m)}$, there exists a $G$ such that $G^{(m)}$ computes a $1$-random over $D^{(m)}$ and $(Z \oplus G)^{(m)} \not\gg D^{(m)}$.
\end{theorem}

\begin{theorem}
    \label{thmforWWKLmvsDNRm}
    For any $D$ and any $Z$ such that $Z^{(m)}$ is not $(x \mapsto 2^x)$-DNC over $D^{(m)}$, there exists a $G$ such that $G^{(m)}$ is DNC over $D^{(m)}$ and $(Z \oplus G)^{(m)}$ is not $(x \mapsto 2^x)$-DNC over $D^{(m)}$.
\end{theorem}

\begin{theorem}
    \label{thmforDNRmvsWKLn}
    For any $D$, any $Z$ such that $Z^{(m)}$ is not DNC over $D^{(m)}$, and any $A$, there exists a $G$ such that $G^{(m+1)} \geq_\text{T} A$ and $(Z \oplus G)^{(m)}$ is not DNC over $D^{(m)}$.
\end{theorem}

Similarly to the base case, Theorem \ref{thmforDNRmvsWKLn} is stronger than what we need to prove Theorem \ref{DNRmvsWKLn}; not only can Theorem \ref{thmforDNRmvsWKLn} be used to produce a set $G$ whose $n$th jump is of PA degree over $D^{(n)}$ (where $n > m$), it can be used to produce a set $G$ whose $(m+1)$st jump computes any arbitrary fixed set $A$.

The proofs of these three theorems are very similar in form. Here, we only provide an explicit proof of Theorem \ref{thmforDNRmvsWKLn}, but this proof can be modified in a straightforward way to produce proofs of Theorems \ref{thmforWKLmvsWWKLm} and \ref{thmforWWKLmvsDNRm}.

\begin{proof}[Proof of Theorem \ref{thmforDNRmvsWKLn}]
    We prove this by induction on $m$. The base case ($m=0$) is given by Theorem \ref{thmforDNR0vsWKL1}. In the inductive case, we assume that, for any $D$, any $Z$ with $Z^{(m-1)}$ not DNC over $D^{(m-1)}$, and any $A$, there exists some $G$ such that $(Z \oplus G)^{(m-1)}$ is not DNC over $D^{(m-1)}$ and that $G^{(m)} \geq_\text{T} A$. For simplicity of notation, we will prove the inductive step in the case where $D = \emptyset$; the argument easily relativizes by writing the letter ``$D$'' where necessary. By the inductive hypothesis (taken with $D = \emptyset'$ and $Z = Z'$), there exists some $G_0$ such that $(Z' \oplus G_0)^{(m-1)}$ is not DNC over $(\emptyset')^{(m-1)} = \emptyset^{(m)}$ and $G_0^{(m)} \geq_\text{T} A$. To produce the requisite set $G$, we will use a version of Towsner forcing, a poset originally used in \cite{Towsner2015} to prove a conservativity result and which is suitable for producing sets that are computationally weak but whose jumps compute some fixed object. Our \emph{conditions} are pairs $(g,k)$ where $g \subseteq \omega^2 \to 2$ is a finite partial function and $k \in \omega$, and the \emph{interpretation} $[g,k]$ of a condition is the class of all partial functions $h \subseteq \omega^2 \to 2$ such that
    \begin{itemize}
        \item $g \subseteq h$ and
        \item for all $(x,y) \in \text{dom}(h) \setminus \text{dom}(g)$, if $x < k$, then $h(x,y) = G_0(x)$.
    \end{itemize}
    If $(h,\ell)$ is a condition such that $h \in [g,k]$ and $\ell \geq k$, we say that $(h, \ell)$ \emph{extends} $(g,k)$ and we denote this $(h, \ell) \leqslant (g,k)$. Every filter $\mathcal{F}$ for this notion of forcing induces a function $f_\mathcal{F} := \bigcup \{ g : (g,k) \in \mathcal{F} \}$, and, if $\mathcal{F}$ is sufficiently generic, then $f_\mathcal{F}$ is total and $\lim_{x \to \infty} f_\mathcal{F}(x,y) = G_0$, which implies, by the limit lemma, that $f_\mathcal{F}' \geq_\text{T} G_0$ and thus that $f_\mathcal{F}^{(m+1)} \geq_\text{T} G_0^{(m)} \geq_\text{T} A$. Thus, all that remains to show is that $(Z, \oplus f_\mathcal{F})^{(m)}$ is not DNC over $\emptyset^{(m)}$.
    
    To do this, we will define the following ``forcing question,'' which is a special type of completion of the forcing relation. Speaking generally, a forcing question takes as input a forcing condition and a formula, and outputs an answer of ``yes'' or ``no.'' If the condition forces the formula, then the question will output ``yes,'' and if the condition forces the formula's negation, the question will output ``no.'' For all other formulas (i.e., for those formulas neither forced to be true nor forced to be false by the condition), the question will still give an answer. In this way, a forcing question forms a dividing line between two sets of formulas. In addition to being a completion of the forcing relation, forcing questions will generally also satisfy some other special properties depending on which weakness property one is trying to preserve. For more on forcing questions, see \cite{patey_lowness_avoidance}.

    \begin{definition}
        \label{forcingquestion}
        Fix a condition $(g,k)$ and a $\Sigma^0_i(Z)$ formula $\varphi(f) \equiv \exists x \psi(f,x)$ (where $\psi(f,x)$ is a $\Pi^0_{i-1}(Z)$ formula).
        \begin{enumerate}
            \item If $i = 1$, let $(g,k) \ {?}{\vdash} \ \varphi(f)$ hold if there is some $x$ and some $(h, \ell) \leqslant (g,k)$ such that $\psi(h,x)$ holds.
            \item If $i > 1$, let $(g,k) \ {?}{\vdash} \ \varphi(f)$ hold if there is some $x$ and some $(h, \ell) \leqslant (g,k)$ such that $(h, \ell) \ {?}{\nvdash} \ \neg \psi(f,x)$.
        \end{enumerate}
        For a $\Pi^0_i(Z)$ formula $\varphi(f)$, define $(g,k) \ {?}{\vdash} \ \varphi(f)$ to hold if $(g,k) \ {?}{\nvdash} \ \neg \varphi(f)$.
    \end{definition}
     By definition, for any $(g,k)$ and any formula $\varphi(f)$ arithmetical in $Z$, either $(g,k) \ {?}{\vdash} \ \varphi(f)$ holds or $(g,k) \ {?}{\vdash} \ \neg \varphi(f)$ holds. Thus, this ``forcing question'' is a completion of the forcing relation. It also enjoys the following additional special proeprties.
    \begin{claim}
        Let $\varphi(f)$ be a $\Sigma^0_i(Z)$ formula.
        \begin{enumerate}
            \item If $(g,k) \ {?}{\vdash} \ \varphi(f)$, then there is an extension $(h,\ell) \leqslant (g,k)$ forcing $\varphi(f)$.
            \item If $(g,k) \ {?}{\nvdash} \ \varphi(f)$, then $(g,k)$ forces $\neg \varphi(f)$.
            \item Fix a condition $(g,k)$. If $\varphi(f)$ is $\Sigma^0_1(Z)$, then the statement ``$(g,k) \ {?}{\vdash} \ \varphi(f)$'' is $\Sigma^0_1(Z)$. On the other hand, if $\varphi(f)$ is $\Sigma^0_i(Z)$ for $i > 1$, then the statement ``$(g,k) \ {?}{\vdash} \ \varphi(f)$'' is $\Sigma^0_1((Z' \oplus G_0)^{(i-2)})$.
        \end{enumerate}
    \end{claim}

    \begin{proof}[Proof of Claim]
        We prove parts (1) and (2) simultaneously by induction on $i$. In the base case, assume that $\varphi(f) \equiv \exists x \psi(f,x)$, where $\psi(f,x)$ is $\Delta^0_0(Z)$. Towards proving part (1), assume that $(g,k) \ {?}{\vdash} \ \varphi(f)$. Take $x_0$ and $(h_0, \ell_0) \leqslant (g,k)$ to witness this fact. The condition $(h_0, \ell_0)$ forces $\varphi(f)$, as needed. Towards proving part (2), assume that $(g,k)$ does not force $\neg \varphi(f)$. By the definition of forcing, there exists some $(h, \ell) \leqslant (g,k)$ and some $x$ such that $\psi(h,k)$ holds. By the definition of ${?}{\vdash}$, it follows that $(g,k) \ {?}{\vdash} \ \varphi(f)$, as needed.

        In the inductive case, let $i > 1$, and assume that $\varphi(f) \equiv \exists x \psi(f,x)$, where $\psi$ is $\Pi^0_{i-1}(Z)$. Towards proving part (1), assume that $(g,k) \ {?}{\vdash} \ \varphi(f)$. Take $x_0$ and $(h_0, \ell_0) \leqslant (g,k)$ to witness this fact. By part (2) of the inductive hypothesis, $(h_0, \ell_0)$ forces $\varphi(f)$, as needed. Towards proving part (2), assume that $(g,k) \ {?}{\nvdash} \ \varphi(f)$. By the definition of ${?}{\vdash}$, it follows that, for every $x$ and every $(h_0, \ell_0) \leqslant (g,k)$, we have $(h_0, \ell_0) \ {?}{\nvdash} \ \psi(f,x)$. By the inductive hypothesis, it follows that, for every $x$ and every $(h_0, \ell_0) \leqslant (g,k)$, there is some $(h, \ell) \leqslant (h_0, \ell_0)$ forcing $\neg \psi(f,kx)$. Therefore, for each sufficiently generic filter $\mathcal{F}$ containing $(g,k)$ and each $x$, there is some $(h,\ell) \in \mathcal{F}$ forcing $\neg \psi(f,x)$. Thus, $(g,k)$ forces $\neg \varphi(f)$.

        For part (3), we prove the case $i = 1$ directly, and then we prove the case $i > 1$ by induction on $i$. First, let $i = 1$. The statement ``$(g,k) \ {?}{\vdash} \ \varphi(f)$'' asserts the existence of a number $x$ and a pair $(h,\ell)$ (each of which is a finitary object) satisfying a predicate that is computable in $Z$ together with the first $k$ bits of the set $G_0$. Since determining the first $k$ bits of $G_0$ is computable (although not uniformly in $k$), the statement ``$(g,k) \ {?}{\vdash} \ \varphi(f)$'' is $\Sigma^0_1(Z)$.
        
        Now, we will prove the case $i > 1$ by induction. In the base case, let $i = 2$. By definition, the statement ``$(g,k) \ {?} {\vdash} \ \varphi(f)$'' asserts that there is some $x$ and some $(h, \ell) \leqslant (g,k)$ such that $(h, \ell) \ {?}{\nvdash} \ \neg \psi(f,x)$, where $\psi(f,x)$ is some $\Pi^0_1(Z)$ statement such that $\varphi(f) = \exists x \psi(f,x)$. It is tempting to argue that, since the quantifications over $x$ and $(h,\ell)$ only use the first $k$ bits of $G_0$ and since the statement ``$(h, \ell) \ {?}{\nvdash} \ \neg \psi(f,x)$'' is $\Sigma^0_1(Z)$, it follow thats the statement ``$(g,k) \ {?}{\vdash} \ \varphi(f)$'' has complexity $\Sigma^0_2(Z)$. However, such an argument would not be correct, because, although the statement ``$(h, \ell) \ {?}{\nvdash} \ \neg \psi(f,x)$'', while $\Sigma^0_1(Z)$, is not $\Sigma^0_1(Z)$ \emph{uniformly} in $(h, \ell)$. However, the statement ``$(g,k) \ {?}{\vdash} \ \varphi(f)$'' does have complexity $\Sigma^0_1(Z' \oplus G_0)$. To prove this, we will present a certain precise formulation of $\varphi(f)$ whose complexity is straightforward to analyze. Assume that $\varphi(f) \equiv \exists x \forall y \chi(f,x,y)$, where $\chi(f,x,y)$ is $\Delta^0_0(Z)$.
        \begin{align*}
            (g,k) \ {?}{\vdash} \ \varphi(f) \ \equiv \ & \exists x \ \exists \text{ finite partial } h \subseteq \omega^2 \to 2 \\
            & \ \ \ \ \ \ \ \ [\forall x < k, \ \forall y, \ (x,y) \in \text{dom}(h_0) \ \to \ h_0(x,y) = G_0(x)], \\ & \exists \ell \geq k, \ \exists F : 2^\ell \to 2 \ \ [\forall x < \ell \ F(x) = G_0(x)], \\
            & \ \ \ \ \forall y \  \forall \text{ finite partial } h_0 \subseteq \omega^2 \to 2 \\
            & \ \ \ \ \ \ \ \ \ \ [\forall x < \ell, \ \forall y, \ (x,y) \in \text{dom}(h_0) \ \to \ h_0(x,y) = F(x)], \\
            & \ \ \ \ \ \ \ \ \chi(h_0,x,y).
        \end{align*}
        The important feature of this formulation of $\varphi(f)$ is that its second half (i.e., the subformula starting with ``$\forall y$...'') is uniformly $\Pi^0_1$ in $Z$ \emph{as well as} in all the variables quantified over in the previous part of the formula. From this observation, we can conclude that the statement ``$(g,k) \ {?}{\vdash} \ \varphi(f)$ is $\Sigma^0_1(Z' \oplus G_0)$.

        To prove the inductive case ($i > 2$), one can construct a precise formulation similar to the one above. The key idea behind this formulation is for each block of quantifiers to include a finite set $F$ encoding however many bits of information about $G_0$ are needed to make the subsequent subformula uniformly definable in all of the variables that have yet been introduced.
    \end{proof}

    Now, we return to proving that $(Z, \oplus f_\mathcal{F})^{(m)}$ is not DNC over $\emptyset^{(m)}$. For each condition $(g,k)$ and each Turing index $e$, we must find an extension $(h, \ell) \leqslant (g,k)$ forcing $\Phi_e^{(Z \oplus f)^{(m)}}$ not to be a DNC function over $\emptyset^{(m)}$. Fix a condition $(g,k)$, and define $$\mathcal{U} := \{ (x,v) \in \omega^2 : (g,k) \ {?}{\vdash} \ \varphi(f) \},$$ where $\varphi(f)$ is some $\Sigma^0_{m+1}(Z)$ formula asserting that $\Phi_e^{(Z \oplus f)^{(m)}}(x) \downarrow v$. This set $\mathcal{U}$ is the set of possible input-output pairs for the function $\Phi_e^{(Z \oplus f)^{(m)}}$ that are validated by the forcing question for this condition. There are three cases:
    \begin{itemize}
        \item Case 1: There is some $x$ such that $(x, \Phi_x^{\emptyset^{(m)}}(x)) \in \mathcal{U}$. In this case, by part (1) of our Claim, we can find an extension $(h, \ell) \leqslant (g,k)$ forcing $\Phi_e^{(Z \oplus f)^{(m)}}(x) \downarrow = \Phi_x^{\emptyset^{(m)}}(x)$.
        \item Case 2: There is some $x$ such that, for each $v \in \omega$, $(x,v) \notin \mathcal{U}$. In this case, by part (2) of our Claim, $(g,k)$ itself forces $\Phi_e^{(Z \oplus f)^{(m)}}(x) \uparrow$.
        \item Case 3: Neither of the above holds. In this case, define $\mathcal{U}_0$ to be the set of pairs $(x,v)$ such that $(x,v) \in \mathcal{U}$ and $(x,w) \notin \mathcal{U}$ for all $w < v$. This set $\mathcal{U}_0$ is computable in $\mathcal{U}$. By part (3) of our claim, the set $\mathcal{U}$ is of complexity $\Sigma^0_{1}((Z' \oplus G_0)^{(m-1)})$. It follows that $\mathcal{U}_0$ is also of complexity $\Sigma^0_1((Z' \oplus G_0)^{(m-1)})$. Since neither Case 1 nor Case 2 holds, $\mathcal{U}_0$ is the graph of a $\emptyset^{(m)}$-DNC function, contradiction our assumption that $(Z' \oplus G_0)^{(m-1)}$ is not DNC over $\emptyset^{(m)}$.
    \end{itemize}
\end{proof}

\printbibliography

\end{document}